\documentclass[12pt, reqno]{amsart}
\usepackage{amsmath, amsthm, amscd, amsfonts, amssymb, graphicx, color}
\usepackage[bookmarksnumbered, colorlinks, plainpages]{hyperref}
\hypersetup{colorlinks=true,linkcolor=red, anchorcolor=green, citecolor=cyan, urlcolor=red, filecolor=magenta, pdftoolbar=true}

\newtheorem{theorem}{Theorem}[section]
\newtheorem{lemma}[theorem]{Lemma}
\newtheorem{proposition}[theorem]{Proposition}
\newtheorem{corollary}[theorem]{Corollary}
\theoremstyle{definition}
\newtheorem{definition}[theorem]{Definition}
\newtheorem{example}[theorem]{Example}

\theoremstyle{remark}
\newtheorem{remark}[theorem]{Remark}
\numberwithin{equation}{section}

\begin{document}

	\setcounter{page}{1}
	
	\title[$K$-$g$-frames in Hilbert module over locally-$C^*$-algebras ]{$K$-$g$-frames in Hilbert module over locally-$C^*$-algebras}
	
	\author{Roumaissae Eljazzar$^{*}$, Mohammed Mouniane \MakeLowercase{and} Mohamed Rossafi}
	
	\address{Department of Mathematics, University of Ibn Tofail, Kenitra, Morocco}
	\email{roumaissae.eljazzar@uit.ac.ma; mohammed.mouniane@uit.ac.ma; mohamed.rossafi@uit.ac.ma}
	
	
	
	\date{$^{*}$Corresponding author}
	
	\setcounter{page}{1}

	\begin{abstract}
		
		This paper explores the concept of $K$-$g$-frames in locally $C^*$-algebras, which are shown to be more general than $g$-frames. The authors first introduce the notion of a $g$-orthonormal basis and utilize it to define the $g$-operator, a crucial element for studying the construction of $K$-$g$-frames in locally $C^*$-algebras. The paper establishes a relationship between $g$-frames and $K$-$g$-frames and introduces the $K$-dual $g$-frame along with its properties. Finally, the authors characterize $K$-$g$-frames through two other related concepts.
	\end{abstract}

	\maketitle
	
		\section{Introduction and preliminaries }

	The concept of Locally $C^*$-algebras, also referred to as pro-$C^*$-algebras, represents an extension of the $C^*$-algebras framework. The notion of Locally $C^*$-algebras was first introduced in the scholarly literature in 1971 by A. Inoue \cite{Inoue}. These algebras have been studied under various terminologies, including pro-$C^*$-algebras as mentioned by D. Voiculescu and N.C. Philips, multinormed $C^*$-algebras as described by A. Dosiev, and $LMC^*$-algebras as identified in the works of G. Lassner and K. Schm\"{u}dgen.

	Hilbert pro-$C^*$-modules can be analogized to Hilbert spaces, with the distinction that the inner product values are in a pro-$C^*$-algebra rather than in the complex number field. For a comprehensive overview, readers are directed to the following papers: \cite{Rouma1, Frago, Frag, Inoue, Mallios, Philip, Philips}.

	A pro-$C^{\ast}$-algebra is viewed as a complete Hausdorff complex topological $\ast$-algebra, denoted as $\mathcal{A}$, where its topology is defined by its continuous $C^{\ast}$-seminorms. Specifically, a net $\{a_{\alpha}\}$ is said to converge to $0$ if and only if $p_{\alpha}(a_{\alpha})$ converges to $0$ for every continuous $C^{\ast}$-seminorm $p_{\alpha}$ on $\mathcal{A}$. Furthermore, it is characterized by the following property:
	
	\begin{enumerate}
		\item[1)] $p_{\alpha}(\gamma \beta) \leq p_{\alpha}(\gamma)p_{\alpha}(\beta)$
		\item[2)] $p_{\alpha}(\gamma^{\ast}\gamma)=p_{\alpha}(\gamma)^{2}$.
	\end{enumerate}

	For every $\gamma, \beta \in \mathcal{A}$, the term $sptr(\gamma)$ refers to the spectrum of $\gamma$, defined as: $sptr(\gamma) = \left\{ \lambda \in \mathbb{C} : \lambda 1_{\mathcal{A}} - \gamma \right\}$ is not invertible, for all $\gamma \in \mathcal{A}$. Here, $\mathcal{A}$ represents a unital pro-$C^{*}$-algebra with unity $1_{\mathcal{A}}$.
	
	The notation $Se(\mathcal{A})$ is used to denote the set of all continuous $C^{\ast}$-seminorms on $\mathcal{A}$. Furthermore, $\mathcal{A}^{+}$ signifies the set of all positive elements in $\mathcal{A}$, which forms a closed convex subset under the $C^{*}$-seminorms on $\mathcal{A}$.

	$\mathcal{H}_{\mathcal{A}}$  denotes the set of all sequences  $\left(\gamma_{n}\right)_{n}$ with $\gamma_{n} \in \mathcal{A}$ such that $\sum_{n} \gamma_{n}^{*} \gamma_{n}$ converges in $\mathcal{A}$. 
	\begin{example}
		Every $C^{*}$-algebra qualifies as a locally $C^{*}$-algebra.
	\end{example}

	\begin{definition}[\cite{Philips}]
		A pre-Hilbert module over a locally $C^{\ast}$-algebra $\mathcal{A}$ is defined as a complex vector space $\mathcal{U}$ that simultaneously functions as a left $\mathcal{A}$-module, aligning with the complex algebra structure. This module is endowed with an $\mathcal{A}$-valued inner product $\langle .,.\rangle: \mathcal{U} \times \mathcal{U} \rightarrow \mathcal{A}$, which exhibits $\mathbb{C}$-linearity and $\mathcal{A}$-linearity in its first argument and fulfills the following criteria:
		\begin{enumerate}
			\item[1)] For all $\xi, \eta \in \mathcal{U}$, it holds that $\langle \xi, \eta\rangle^{*}=\langle \eta, \xi\rangle$.
			\item[2)] For every $\xi \in \mathcal{U}$, the condition $\langle \xi, \xi\rangle \geq 0$ is satisfied.
			\item[3)] The equality $\langle \xi, \xi\rangle = 0$ is true if and only if $\xi = 0$.
		\end{enumerate}
		The space $\mathcal{U}$ is referred to 
		Hilbert $\mathcal{A}$-module (or Hilbert pro-$C^{\ast}$-module over $\mathcal{A}$) when it achieves completeness in relation to the topology that emerges from the family of seminorms defined by:
		$$
		\bar{p}_{\mathcal{U}}(\xi) = \sqrt{p_{\alpha}(\langle \xi, \xi\rangle)} \quad \forall \xi \in \mathcal{U}, p \in Se(\mathcal{A}).
		$$
	\end{definition}

	Throughout the remainder of this document, we consider $\mathcal{A}$ as a pro-$C^{\ast}$-algebra. Furthermore, $\mathcal{U}$ and $\mathcal{V}$ are designated as Hilbert modules over the algebra $\mathcal{A}$, with $I$ and $J$ representing countably infinite sets of indices.

	An operator from $\mathcal{U}$ to $\mathcal{V}$ is defined as any bounded $\mathcal{A}$-module map from $\mathcal{U}$ to $\mathcal{V}$. The collection of all such operators from $\mathcal{U}$ to $\mathcal{V}$ is denoted by $Hom_{\mathcal{A}}(\mathcal{U}, \mathcal{V})$.
	
	\begin{definition}[\cite{Azhini}]
		Consider $\mathfrak{F}: \mathcal{U} \rightarrow \mathcal{V}$ as an $\mathcal{A}$-module map. The operator $\mathfrak{F}$ is termed adjointable if there exists a map $\mathfrak{F}^{\ast}: \mathcal{V} \rightarrow \mathcal{U}$ satisfying $\langle \mathfrak{F} \xi, \eta\rangle = \langle \xi, \mathfrak{F}^{\ast} \eta\rangle$ for all $\xi \in \mathcal{U}$, $\eta \in \mathcal{V}$. $\mathfrak{F}$ is considered bounded if, for all $p \in Se(\mathcal{A})$, there exists $N_{p} > 0$ such that 
		$$\bar{p_{\alpha}}_{\mathcal{V}}(\mathfrak{F} \xi) \leq N_{p} \bar{p_{\alpha}}_{\mathcal{U}}(\xi), \; \forall \xi \in \mathcal{U}.$$ 
		
		The set of all adjointable operators from $\mathcal{U}$ to $\mathcal{V}$ is represented by $Hom_{\mathcal{A}}^{\ast}(\mathcal{U}, \mathcal{V})$, and $Hom_{\mathcal{A}}^{\ast}(\mathcal{U}) = Hom_{\mathcal{A}}^{\ast}(\mathcal{U}, \mathcal{U})$.
	\end{definition}

	\begin{definition}[\cite{Alizadeh}]
		Consider the operator $\mathfrak{F}: \mathcal{U} \rightarrow \mathcal{V}$. We say it is uniformly bounded below if there exists a constant $A>0$ such that for each $p_{\alpha} \in Se(\mathcal{A})$,
		\begin{equation*}
			\bar{p_{\alpha}}_{\mathcal{V}}(\mathfrak{F} \xi) \leqslant A \bar{p_{\alpha}}_{\mathcal{U}}(\xi), \quad \forall \xi \in \mathcal{U}.
		\end{equation*}
		Similarly, it is termed uniformly bounded above if a constant $A^{\prime}>0$ ensures that for every $p \in Se(\mathcal{A})$,
		\begin{equation*}
			\bar{p_{\alpha}}_{\mathcal{V}}(\mathfrak{F} \xi) \geqslant A^{\prime} \bar{p_{\alpha}}_{\mathcal{U}}(\xi), \quad \forall \xi \in \mathcal{U}
		\end{equation*}
		\begin{equation*}
			\|\mathfrak{F}\|_{\infty}=\inf \{N: N \text { is an upper bound for } \mathfrak{F}\}
		\end{equation*}
		\begin{equation*}
			\hat{p}_{\mathcal{V}}(\mathfrak{F})=\sup \left\{\bar{p_{\alpha}}_{\mathcal{V}}(\mathfrak{F}(\xi)): \xi \in \mathcal{U}, \quad \bar{p_{\alpha}}_{\mathcal{U}}(\xi) \leqslant 1\right\}.
		\end{equation*}
		It is clearly evident that, $\hat{p}(\mathfrak{F}) \leqslant\|\mathfrak{F}\|_{\infty}$ for all $p \in Se(\mathcal{A})$.   
	\end{definition}

	\begin{proposition}[\cite{Azhini}] \label{Prop2.6}
		Assume $\mathfrak{F}$ belongs to $Hom_{\mathcal{A}}^{\ast}(\mathcal{U})$ and is an invertible operator, with both $\mathfrak{F}$ and its inverse $\mathfrak{F}^{-1}$ being uniformly bounded. For any element $\eta \in \mathcal{U}$, the following inequality holds:
		$$
		\left\|\mathfrak{F}^{-1}\right\|_{\infty}^{-2}\langle \eta, \eta \rangle \leq \langle \mathfrak{F} \eta, \mathfrak{F} \eta \rangle \leq \|\mathfrak{F}\|_{\infty}^{2}\langle \eta, \eta \rangle.
		$$
	\end{proposition}
	
	\begin{lemma}[\cite{Rouma}]\label{Dougl1}
		In the setting where $\mathcal{U}$ is a Hilbert $\mathcal{A}$-module within a pro-$C^{*}$-algebra $\mathcal{A}$, and given $T, Z \in Hom_{\mathcal{A}}^{*}(\mathcal{U})$, the following conditions are equivalent when the range of $Z$, denoted as $Ran(Z)$, is closed:
		\begin{enumerate}
			\item $Ran(T)\subseteq Ran(Z)$.
			\item The relation $T T^{*} \leq \alpha^{2} Z Z^{*}$ is valid for some non-negative $\alpha$.
			\item An operator $U$, uniformly bounded and within $Hom_{\mathcal{A}}^{*}(\mathcal{U})$, exists such that $T=Z U$.
		\end{enumerate}
	\end{lemma}

	\begin{definition}[\cite{Haddad}]
		We define a sequence $\Gamma=\left\{\Gamma_{i} \in Hom_{A}^{*}\left(\mathcal{U}, \mathcal{V}_{i}\right)\right\}_{i \in I}$ as a $g$-frame for $\mathcal{U}$ relative to the set $\left\{\mathcal{V}_{i}\right\}_{i \in I}$ if there exist positive constants $A$ and $B$ satisfying the following condition for every $\eta \in \mathcal{U}$:
		
		\begin{equation}\label{eqq}
			A\langle \eta, \eta \rangle \leq \sum_{i \in I}\left\langle\Gamma_{i} \eta, \Gamma_{i} \eta \right\rangle \leq B\langle \eta, \eta \rangle.
		\end{equation}
		
		Here, $A$ and $B$ are referred to as the $g$-frame bounds for $\Gamma$. The $g$-frame is termed tight if $A=B$, and Parseval if $A=B=1$. If only the upper bound is satisfied in \ref{eqq}, $\Gamma$ is identified as a $g$-Bessel sequence.
	\end{definition}
	
	\begin{lemma}[\cite{Haddad}]\label{Lemma2.3}
		A sequence $\left\{\Gamma_i \in Hom_{\mathcal{A}}^*\left(\mathcal{U}, \mathcal{V}_i\right)\right\}_{i \in I}$ forms a $g$-frame for $\mathcal{U}$ in relation to $\left\{\mathcal{V}_i\right\}_{i \in I}$ if and only if
		$$
		Q:\left\{g_i\right\}_{i \in I} \mapsto \sum_{i \in I} \Gamma_i^* g_i,
		$$
		can be characterized as a well-defined and bounded linear operator mapping from $\bigoplus_{i \in I} \mathcal{V}_i$ onto $\mathcal{U}$.
	\end{lemma}

	\section{$K$-$g$-frame in Hilbert Modules over Pro-$C^*$-Algebras}
	
	\begin{definition}
		Given $K \in Hom_{\mathcal{A}}^{*}(\mathcal{U})$, a sequence $\Gamma=\left\{\Gamma_{i} \in Hom_{A}^{*}\left(\mathcal{U}, \mathcal{V}_{i}\right)\right\}_{i \in I}$ is defined as a $K$-g-frame for $\mathcal{U}$ with regard to $\left\{\mathcal{V}_{i}\right\}_{i \in I}$ if there exist two positive constants $C$ and $D$ such that for all $\xi \in \mathcal{U}$, the following inequality is satisfied:
		$$
		C\langle K^{*} \xi, K^{*}\xi \rangle \leq \sum_{i \in I}\left\langle\Gamma_{i} \xi, \Gamma_{i} \xi \right\rangle \leq D\langle \xi, \xi \rangle.
		$$
		The constants $C$ and $D$ are known as the lower and upper bounds of the $K$-g-frame, respectively. If $C=D$, the $K$-g-frame is termed tight, and it is known as a Parseval frame when $C=D=1$.
	\end{definition}

	\begin{definition}
		Consider $\Gamma=\left\{\Gamma_i \in Hom_\mathcal{A}^*\left(\mathcal{U}, \mathcal{V}_i\right)\right\}_{i \in I}$ constituting a $K$-g-frame for $\mathcal{U}$ in relation to $\left\{\mathcal{V}_i\right\}_{i \in I}$, equipped with an analysis operator $T$. The frame operator $S: \mathcal{U} \rightarrow \mathcal{U}$ is then defined as follows:
		$$
		S \xi=T^* T \xi=\sum_{i \in I} \Gamma_i^* \Gamma_i \xi, \quad \xi \in \mathcal{U}.
		$$
	\end{definition}
	
	\begin{definition}
		A sequence $\left\{\Gamma_{i} \in Hom_{\mathcal{A}}^{*}(\mathcal{U},\mathcal{V}_{i}): i \in I\right\}$ is termed a g-orthonormal basis for $\mathcal{U}$ with respect to $\left\{\mathcal{V}_{i}\right\}_{i \in I}$ if it fulfills the following conditions:
		$$
		\begin{aligned}
			\left\langle\Gamma_{i}^{*} g_{i}, \Gamma_{j}^{*} g_{j}\right\rangle &=\delta_{i j}\left\langle g_{i}, g_{j}\right\rangle, \quad \forall i, j \in I, g_{i} \in \mathcal{V}_{i}, g_{j} \in \mathcal{V}_{j}, \\
			\sum_{i \in I}\bar{p_{\alpha}}_{\mathcal{U}}(\Gamma_{i} \xi)^{2} &=\bar{p_{\alpha}}_{\mathcal{U}}( \xi)^{2}, \quad \forall \xi \in \mathcal{U}.
		\end{aligned}
		$$
	\end{definition}

	\begin{definition}
		Given $K \in Hom_{\mathcal{A}}^*(\mathcal{U})$ and a $K$-g-frame $\left\{\Gamma_i\right\}_{i \in I}$ for $\mathcal{U}$ with respect to $\left\{\mathcal{V}_i\right\}_{i \in I}$, a $g$-frame sequence $\left\{\Xi_i\right\}_{i \in I}$ for $\mathcal{U}$ with respect to $\left\{\mathcal{V}_i\right\}_{i \in I}$ is considered a $K$-dual g-frame sequence of $\left\{\Gamma_i\right\}_{i \in I}$ if the following holds true:
		$$
		K \xi=\sum_{i \in I} \Gamma_i^* \Xi_i \xi, \quad \forall \xi \in \mathcal{U}.
		$$
	\end{definition}

	\begin{lemma}\label{Lemma2.5}
		Assume $\left\{\Gamma_i\right\}_{i \in I}$ as a g-Bessel sequence within $\mathcal{U}$, associated with $\left\{\mathcal{V}_i\right\}_{i \in I}$. The sequence $\left\{\Gamma_i\right\}_{i \in I}$ forms a $K$-g-frame for $\mathcal{U}$ with respect to $\left\{\mathcal{V}_i\right\}_{i \in I}$ if there is a constant $C>0$ such that the following condition is met: $S \geq C K K^*$, with $S$ being the frame operator for $\left\{\Gamma_i\right\}_{i \in I}$.
	\end{lemma}

	\begin{proof}
		The sequence $\left\{\Gamma_i\right\}_{i \in I}$ qualifies as a $K$-g-frame for $\mathcal{U}$ with respect to $\left\{\mathcal{V}_i\right\}_{i \in I}$ with bounds $C$, $D$ if and only if the following condition is satisfied:
		$$
		C\langle K^* \xi, K^{*} \xi \rangle \leq \sum_{i \in I}\langle \Gamma_i \xi, \Gamma_i \xi\rangle \leq D \langle \xi, \xi \rangle, \forall \xi \in \mathcal{U},
		$$
		which can be equivalently expressed as:
		$$
		\left\langle C K K^* \xi, \xi \right\rangle \leq \langle S \xi, \xi\rangle \leq \langle D \xi, \xi\rangle, \forall \xi \in \mathcal{U},
		$$
		where $S$ denotes the frame operator of the $K$-g-frame $\left\{\Gamma_i\right\}_{i \in I}$. Thus, the assertion is substantiated.
	\end{proof}

	\begin{definition}
		Consider $\Gamma_i \in Hom_{\mathcal{A}}^*(\mathcal{U},\mathcal{V}_{i})$.
		The sequence $\left\{\Gamma_i\right\}_{i \in I}$ is deemed $g$-complete if the set $\left\{\xi \in \mathcal{U}: \Gamma_i \xi = 0 \text{ for all } i \in I\right\}=\{0\}$.
	\end{definition}

	\begin{proposition}\label{prop3.6}
		The set $\left\{\Gamma_{i} \in Hom_{\mathcal{A}}^*(\mathcal{U},\mathcal{V}_{i}): i \in I\right\}$ attains g-completeness if and only if \\ $$\overline{\operatorname{span}}\left\{\Gamma_{i}^{*}\left(\mathcal{V}_{i}\right)\right\}_{i \in I}=\mathcal{U}.$$
	\end{proposition}

	\begin{proof}
		Assume $\left\{\Gamma_{i} \in Hom_{\mathcal{A}}^*(\mathcal{U},\mathcal{V}_{i}): i \in I\right\}$ to be $g$-complete. Given that $\overline{\operatorname{span}}\left\{\Gamma_{i}^{*}\left(\mathcal{V}_{i}\right)\right\}_{i \in I} \subseteq \mathcal{U}$, it suffices to show that if $\zeta \in \mathcal{U}$ and $\zeta \in \operatorname{span}\left\{\Gamma_{i}^{*}\left(\mathcal{V}_{i}\right)\right\}_{i \in I}$, then $\zeta=0$. Let $\zeta \in \mathcal{U}$ and $\zeta \perp \operatorname{span}\left\{\Gamma_{i}^{*}\left(\mathcal{V}_{i}\right)\right\}_{i \in I}$. For any $i \in I, \zeta \perp \Gamma_{i}^{*} \Gamma_{i}(\zeta)$ implies that for all $i \in I$,
		$$
		\bar{p_{\alpha}}_{\mathcal{U}}(\Gamma_{i}\zeta)^{2}=p_{\alpha}(\left\langle \zeta, \Gamma_{i}^{*} \Gamma_{i}(\zeta)\right\rangle)=0 .
		$$
		Therefore, by $g$-completeness of $\left\{\Gamma_{i} \in Hom_{\mathcal{A}}^*(\mathcal{U},\mathcal{V}_{i}): i \in I\right\}$, it follows that $\zeta=0$. \\
		Conversely, if $\overline{\operatorname{span}}\left\{\Gamma_{i}^{*}\left(\mathcal{V}_{i}\right)\right\}_{i \in I}=\mathcal{U}$, and for a given $\zeta \in \mathcal{U}$, $\Gamma_{i} \zeta=0$ for all $i \in I$, then for each $\theta \in \mathcal{V}_{i}$
		$$
		\left\langle\Gamma_{i} \zeta, \theta \right\rangle=\left\langle \zeta, \Gamma_{i}^{*} \theta\right\rangle=0,
		$$
		leading to $\zeta \perp \operatorname{span}\left\{\Gamma_{i}^{*}\left(\mathcal{V}_{i}\right)\right\}_{i \in I}$. Thus, $\zeta \perp \overline{\operatorname{span}}\left\{\Gamma_{i}^{*}\left(\mathcal{V}_{i}\right)\right\}_{i \in I}=\mathcal{U}$. This implies that $\zeta=0$, establishing the $g$-completeness of $\left\{\Gamma_{i} \in Hom_{\mathcal{A}}^*\left(\mathcal{U}, \mathcal{V}_{i}\right): i \in I\right\}$.
	\end{proof}

	\begin{proposition}\label{prop3.6}
		A necessary and sufficient condition for the set $\left\{\Gamma_{i} \in Hom_{\mathcal{A}}^*(\mathcal{U},\mathcal{V}_{i}): i \in I\right\}$ to be g-complete is that 
		$$\overline{\operatorname{span}}\left\{\Gamma_{i}^{*}\left(\mathcal{V}_{i}\right)\right\}_{i \in I} = \mathcal{U}.$$
	\end{proposition}

	\begin{proof}
		Let's start by assuming that the set $\left\{\Gamma_{i} \in Hom_{\mathcal{A}}^*(\mathcal{U},\mathcal{V}_{i}): i \in I\right\}$ qualifies as g-complete. Noting that $\overline{\operatorname{span}}\left\{\Gamma_{i}^{*}\left(\mathcal{V}_{i}\right)\right\}_{i \in I}$ is a subset of $\mathcal{U}$, we need to demonstrate that any element $\xi \in \mathcal{U}$, which is orthogonal to $\operatorname{span}\left\{\Gamma_{i}^{*}\left(\mathcal{V}_{i}\right)\right\}_{i \in I}$, must necessarily be zero i.e $\xi=0$. For every such element $\xi$ in $\mathcal{U}$ and for all $i \in I$, the orthogonality of $\xi$ to $\Gamma_{i}^{*} \Gamma_{i}(\xi)$ leads to 
		$$
		\bar{p_{\alpha}}_{\mathcal{U}}(\Gamma_{i}\xi)^{2} = p_{\alpha}(\left\langle \xi, \Gamma_{i}^{*} \Gamma_{i}(\xi)\right\rangle) = 0.
		$$
		This, in light of the g-completeness of $\left\{\Gamma_{i} \in Hom_{\mathcal{A}}^*(\mathcal{U},\mathcal{V}_{i}): i \in I\right\}$, implies that $\xi=0$.
		
		On the flip side, if $\overline{\operatorname{span}}\left\{\Gamma_{i}^{*}\left(\mathcal{V}_{i}\right)\right\}_{i \in I}$ is equal to $\mathcal{U}$ and for a particular $\xi \in \mathcal{U}$, we find that $\Gamma_{i} \xi$ equals zero for all $i \in I$, then for any $\eta$ belonging to $\mathcal{V}_{i}$, we have
		$$
		\left\langle\Gamma_{i} \xi, \eta \right\rangle = \left\langle \xi, \Gamma_{i}^{*} \eta\right\rangle = 0,
		$$
		which means $\xi$ is orthogonal to $\operatorname{span}\left\{\Gamma_{i}^{*}\left(\mathcal{V}_{i}\right)\right\}_{i \in I}$, and consequently to $\overline{\operatorname{span}}\left\{\Gamma_{i}^{*}\left(\mathcal{V}_{i}\right)\right\}_{i \in I}=\mathcal{U}$.
		
		This confirms that $\xi=0$, thereby proving the g-completeness of $\left\{\Gamma_{i} \in Hom_{\mathcal{A}}^*\left(\mathcal{U}, \mathcal{V}_{i}\right): i \in I\right\}$.
	\end{proof}

	\begin{lemma}\label{lemma2.7}
		Suppose $\left\{\mathfrak{E}_{i} \in Hom_{\mathcal{A}}^*\left(\mathcal{U}, \mathcal{V}_{i}\right): i \in I \right\}$ constitutes a g-orthonormal basis for $\mathcal{U}$  with respect to $\left\{\mathcal{V}_{i}: i \in I \right\}$. In this case,  $\left\{\Gamma_{i} \in Hom_{\mathcal{A}}^*\left(\mathcal{U}, \mathcal{V}_{i}\right): i \in I\right\}$ will be a $g$-frame sequence with respect to $\left\{\mathcal{V}_{i}: i \in I\right\}$ if and only if a singular bounded operator $Q: \mathcal{U} \to \mathcal{U}$ exists such that $\Gamma_{i}=\mathfrak{E}_{i} Q^{*}$ for each $i \in I$.
	\end{lemma}

	\begin{proof}
		$\Rightarrow$ If  $\left\{\mathfrak{E}_{i} \in Hom_{\mathcal{A}}^*\left(\mathcal{U}, \mathcal{V}_{i}\right): i \in I \right\}$ forms a $g$-orthonormal basis for $\mathcal{U}$, then for any element $\xi \in \mathcal{U}$, the set $\left\{\mathfrak{E}_{i} \xi: i \in I \right\}$ is part of $\left(\sum_{i \in I} \oplus \mathcal{V}_{i}\right)_{\mathcal{H}_{\mathcal{A}}}$. Assuming $\left\{\Gamma_{i} \in Hom_{\mathcal{A}}^*\left(\mathcal{U}, \mathcal{V}_{i}\right): i \in I \right\}$ to be a $g$-Bessel sequence, Lemma \ref{Lemma2.3} ensures that the operator $Q$, defined as 
		$$
		Q: \mathcal{U} \longrightarrow \mathcal{U}, \quad Q \xi=\sum_{i \in I} \Gamma_{i}^{*} \mathfrak{E}_{i} \xi
		$$
		is both well-defined and bounded. Utilizing the $g$-orthonormal basis definition, it's clear that $\mathfrak{E}_{i} \mathfrak{E}_{j}^{*} \eta = \delta_{i j} \eta$. Thus,
		$$
		Q \mathfrak{E}_{j}^{*} \eta = \sum_{i \in I} \Gamma_{i}^{*} \mathfrak{E}_{i} \mathfrak{E}_{j}^{*} \eta = \Gamma_{j}^{*} \mathfrak{E}_{j} \mathfrak{E}_{j}^{*} \eta = \Gamma_{j}^{*} \eta
		$$
		for all $\eta \in \mathcal{V}_{j}, j \in I$. Consequently, $Q \mathfrak{E}_{j}^{*} = \Gamma_{j}^{*}$, leading to $\mathfrak{E}_{j} Q^{*} = \Gamma_{j}$ for each $j \in I$. Now, assume $Q_{1}, Q_{2} \in Hom_{\mathcal{A}}^*\left(\mathcal{U}, \mathcal{V}_{i}\right)$ where $\mathfrak{E}_{i} Q_{1}^{*} = \mathfrak{E}_{i} Q_{2}^{*} = \Gamma_{i}$ for all $i \in I$. For any $\xi \in \mathcal{U}, \eta_{i} \in \mathcal{V}_{i}$, it holds that $\left\langle \mathfrak{E}_{i} Q_{1}^{*} \xi, \eta_{i}\right\rangle = \left\langle \mathfrak{E}_{i} Q_{2}^{*} \xi, \eta_{i}\right\rangle$, implying $\left\langle Q_{1}^{*} \xi, \mathfrak{E}_{i}^{*} \eta_{i}\right\rangle = \left\langle Q_{2}^{*} \xi, \mathfrak{E}_{i}^{*} \eta_{i}\right\rangle$. Since $\overline{\operatorname{span}}\left\{\mathfrak{E}_{i}^{*}\left(\mathcal{V}_{i}\right)\right\}_{i \in I} = \mathcal{U}$ (Proposition \ref{prop3.6}), it follows that $Q_{1}^{*} \xi = Q_{2}^{*} \xi$, establishing $Q_{1} = Q_{2}$. Therefore, $Q$ is unique.
		
		$\Leftarrow$ Given $\Gamma_{i} = \mathfrak{E}_{i}Q^{*}$ for all $i \in I$, for any $\xi \in \mathcal{U}$, we have 
		$$\sum_{i \in I}\langle\Gamma_{i} \xi, \Gamma_{i} \xi \rangle = \sum_{i \in I}\langle \mathfrak{E}_{i} Q^{\ast} \xi, \mathfrak{E}_{i} Q^{\ast} \xi \rangle = \langle Q^{\ast} \xi, Q^{\ast} \xi \rangle.$$
		
		Proposition $2.2$ in \cite{Alizadeh} demonstrates that $Q^*$ is bounded below, hence invertible. Following Theorem $3.2$ in \cite{Azhini}, we obtain:
		$$
		\left\|\left(Q^{*}\right)^{-1}\right\|_{\infty}^{-2} \langle \xi, \xi \rangle \leq \left\langle Q^{*} \xi, Q^{*} \xi\right\rangle \leq \left\|Q^{*}\right\|_{\infty}^{2} \langle \xi, \xi \rangle.
		$$
	\end{proof}

	\begin{remark}
		When the set $\left\{\mathfrak{E}_i \in Hom_{\mathcal{A}}^*\left(\mathcal{U}, \mathcal{V}_{i}\right): i \in I\right\}$ forms a $g$-orthonormal basis, the operator $Q$ described in Lemma \ref{lemma2.7} is recognized as the $g$-operator  corresponding to the sequence $\left\{\Gamma_{i} \in Hom_{\mathcal{A}}^*\left(\mathcal{U}, \mathcal{V}_{i}\right): i \in I\right\}$.
	\end{remark}

	\begin{theorem}
		Given $K \in Hom_{\mathcal{A}}^*(\mathcal{U})$ and  $\left\{\Gamma_j\right\}_{i \in I}$ as a $g$-frame sequence for $\mathcal{U}$ with respect to $\left\{\mathcal{V}_i\right\}_{i \in I}$, if $Q$ is the $g$-operator corresponding to $\left\{\Gamma_i\right\}_{i \in I}$, then $\left\{\Gamma_i\right\}_{i \in I}$ qualifies as a $K$-$g$-frame if and only if $Ran(K) \subset Ran(Q)$.
	\end{theorem}

	\begin{proof}
		$\Rightarrow$ Assume that $\left\{\Gamma_i\right\}_{i \in I}$ is a $K$-$g$-frame for $\mathcal{U}$ with respect to $\left\{\mathcal{V}_i\right\}_{i \in I}$. This means there exists a positive constant $C$ such that
		$$
		C\langle K^{\ast}\xi ,K^{\ast}\xi \rangle \leq \sum_{i \in I}  \langle \Gamma_{i} \xi,  \Gamma_{i} \xi\rangle , \quad \forall \xi \in \mathcal{U}.
		$$
		By Lemma \ref{lemma2.7}, $ \Gamma_i=\mathfrak{E}_i Q^*$, where $\left\{\mathfrak{E}_{i} \in Hom_{\mathcal{A}}^*\left(\mathcal{U}, \mathcal{V}_{i}\right): i \in I \right\}$ is the g-orthonormal basis for $\mathcal{U}$ with respect to $\left\{\mathcal{V}_i\right\}_{i \in I}$. Therefore, we obtain
		$$
		C\langle K^{\ast}\xi ,K^{\ast}\xi \rangle \leq \sum_{i \in I}  \langle \mathfrak{E}_i Q^* \xi,  \mathfrak{E}_i Q^* \xi\rangle = \langle  Q^* \xi,  Q^* \xi\rangle, \quad  \forall \xi \in \mathcal{U},
		$$
		which implies that $C K K^* \leq Q Q^*$. So, by Lemma \ref{Dougl1}, we have $Ran(K) \subset Ran(Q)$.
		
		$\Leftarrow$ If $Ran(K) \subset Ran(Q)$, Lemma \ref{Dougl1} suggests the existence of a constant $\alpha>0$ such that $K K^* \leq \alpha Q Q^*$. Then for all $\xi \in \mathcal{U}$,
		$$
		\left\langle\frac{1}{\alpha} K K^* \xi, \xi \right\rangle \leq\left\langle Q Q^* \xi , \xi \right\rangle,
		$$
		that is,
		$$
		\frac{1}{\alpha} \left\langle K^* \xi ,  K^* \xi\right\rangle \leq\left\langle  Q^* \xi , Q^*\xi \right\rangle.
		$$
		Assuming $\left\{\Gamma_i\right\}_{i \in I}$ is a g-frame and $\left\{\mathfrak{E}_{i} \in Hom_{\mathcal{A}}^*\left(\mathcal{U}, \mathcal{V}_{i}\right): i \in I \right\}$ is the g-orthonormal basis for $\mathcal{U}$ in relation to $\{\mathcal{V}_{i}\}_{i \in I}$, then by Lemma \ref{lemma2.7} 
		$$ \langle Q^{\ast} \xi, Q^{\ast} \xi \rangle = \sum_{i \in I} \langle \mathfrak{E}_{i} Q^{\ast} \xi, \mathfrak{E}_{i}Q^{\ast} \xi \rangle = \sum_{i \in I} \langle \Gamma_i \xi, \Gamma_i \xi \rangle $$
		Hence 
		$$ \frac{1}{\alpha} \left\langle K^* \xi ,  K^* \xi\right\rangle \leq \sum_{i \in I} \langle \Gamma_i \xi, \Gamma_i \xi \rangle, \quad \forall \xi \in \mathcal{U} $$ 
		Thus, $\{ \Gamma_i\}_{i \in I}$ is a $K$-$g$-frame.
	\end{proof}

	\begin{theorem}
		Assume $K \in Hom_{\mathcal{A}}^{\ast}(\mathcal{U})$ and consider  $\left\{\Gamma_i\right\}_{i \in I}$ as a $g$-frame sequence for $\mathcal{U}$ with respect to $\left\{\mathcal{V}_i\right\}_{i \in I}$. Let $\left\{\mathfrak{E}_i \in Hom_{\mathcal{A}}^*\left(\mathcal{U}, \mathcal{V}_i\right): i \in I\right\}$ be the g-orthonormal basis for $\mathcal{U}$ with respect to $\left\{\mathcal{V}_i\right\}_{i \in I}$, and let $Q$ be the g-operator corresponding to. $\left\{\Gamma_i\right\}_{i \in I}$. The theorem posits that $Q$ functions as a co-isometry if and only if the sequence $\left\{\Gamma_i K^*\right\}_{i \in I}$ establishes a Parseval $K$-g-frame.
	\end{theorem}

	\begin{proof}
		By utilizing the properties of a $g$-orthonormal basis, we have the following equality for any $\xi \in \mathcal{X}$:
		$$
		\sum_{i \in I}\langle \Gamma_i K^* \xi, \Gamma_i K^* \xi \rangle = \sum_{i \in I}\langle U_i Q^* K^* \xi, U_i Q^* K^* \xi \rangle = \langle Q^* K^* \xi, Q^* K^* \xi \rangle,
		$$
		which clearly demonstrates the conclusion.
	\end{proof}

	\begin{theorem}
		Consider $K \in Hom_{\mathcal{A}}^{\ast}(\mathcal{U})$ and let $\left\{\mathfrak{E}_i \in Hom_{\mathcal{A}}\left(\mathcal{U}, \mathcal{V}_i\right): i \in I\right\}$ constitute a g-orthonormal basis for $\mathcal{U}$  with respect to $\left\{\mathcal{V}_i\right\}_{i \in I}$. Suppose $\left\{\Gamma_i\right\}_{i \in I}$ serves as a $K$-g-frame for $\mathcal{U}$ with respect to $\left\{\mathcal{V}_i\right\}_{i \in I}$, with $Q$ being the corresponding $g$-operator. If $P$ is the $g$-operator for the $g$-frame sequence $\left\{\Xi_i\right\}_{i in I}$ and is invertible, with $P^{-1}$ being the inverse of $Q$, then the sequence $\left\{T_i\right\}_{i in I}$ qualifies as a $K$-g-frame.
	\end{theorem}
	
	\begin{proof}
		Under the assumptions above, there exists a positive constant $C$ such that $\forall \xi \in \mathcal{X}$
		$$
		C\langle K^* \xi ,  K^* \xi  \rangle \leq \sum_{i \in I} \langle \Gamma_{i} \xi , \Gamma_{i}  \xi  \rangle =\sum_{i \in I} \langle U_{i} Q^* \xi , U_{i} Q^*  \xi  \rangle =   \sum_{i \in I} \langle U_{i} P^*\left(P^*\right)^{-1}Q^*  \xi , U_{i} P^*\left(P^*\right)^{-1}Q^*  \xi  \rangle.
		$$
		Since $P$ is invertible and $Q P^{-1}=I$, we obtain $\left(P^*\right)^{-1} Q^*=I$, so
		$$
		C\langle K^* \xi ,  K^* \xi  \rangle \leq \sum_{i \in I} \langle U_i P^* \xi,  U_i P^* \xi \rangle =\sum_{i \in I} \langle \Xi_i \xi, \Xi_i \xi \rangle, \quad \forall \xi \in \mathcal{X}.
		$$
	\end{proof}

	\begin{theorem}
		Assume $K \in Hom_{\mathcal{A}}^{\ast}(\mathcal{U})$. Let us consider a set $\left\{\mathfrak{E}_i \in Hom_{\mathcal{A}}^{\ast}\left(\mathcal{U}, \mathcal{V}_i\right): i \in I\right\}$, constituting a $g$-orthonormal foundation for $\mathcal{U}$ with respect to  $\left\{\mathcal{V}_i\right\}_{i \in I}$. Given that $\left\{\Gamma_i\right\}_{i \in I}$ establishes a $K$-$g$-frame for $\mathcal{U}$ with respect to $\left\{\mathcal{V}_i\right\}_{i \in I}$ with $Q$ denoting the associated $g$-operator, and considering $P$ as the $g$-operator linked to the $g$-frame sequence $\left\{T_i\right\}_{i \in I}$, it follows that $\left\{T_i\right\}_{i \in I}$ forms the $K$-dual $g$-frame sequence of $\left\{\Gamma_i\right\}_{i \in I}$ if and only if the relation $K=Q P^*$ is satisfied.
	\end{theorem}

	\begin{proof}
		Let's assume that $\left\{T_i\right\}_{i \in I}$ is the $K$-dual $g$-frame sequence of $\left\{\Gamma_i\right\}_{i \in I}$. Then for every $\xi \in \mathcal{U}$, it follows that $K \xi = \sum_{i \in I} \Gamma_i^* T_i \xi$. Given that $\left\{\mathfrak{E}_i\right\}_{i \in I}$ forms the $g$-orthonormal basis for $\mathcal{U}$, applying Lemma \ref{lemma2.7} yields that for all $\xi \in \mathcal{U}$,
		$$
		K \xi = \sum_{i \in I}\left(\mathfrak{E}_i Q^*\right)^*\left(\mathfrak{E}_i P^*\right) \xi = Q \sum_{i \in I} \mathfrak{E}_i^* \mathfrak{E}_i P^* \xi = Q P^* \xi.
		$$
		The arbitrariness of $\xi$ leads to the conclusion that $K = Q P^*$.
		
		Conversely, acknowledging that $\left\{\Gamma_i\right\}_{i \in I}$ and $\left\{T_i\right\}_{i \in I}$ are both $g$-frame sequences, Lemma \ref{lemma2.7} indicates the existence of bounded operators $Q$ and $P$ such that
		$$
		\Gamma_i = \mathfrak{E}_i P^*, \quad T_i = \mathfrak{E}_i P^*.
		$$
		Therefore, for each $\xi \in \mathcal{U}$,
		$$
		\sum_{i \in I} \Gamma_i^* T_i \xi = \sum_{i \in I}\left(\mathfrak{E}_i Q^*\right)^*\left(\mathfrak{E}_i P^*\right) \xi = Q \sum_{i \in I} \mathfrak{E}_i^* \mathfrak{E}_i P^* \xi = Q P^* \xi = K \xi.
		$$
		This demonstrates that $\left\{T_i\right\}_{i \in I}$ is indeed the $K$-dual $g$-frame sequence of $\left\{\Gamma_i\right\}_{i \in I}$.
	\end{proof}

	\begin{theorem}
		Consider $K \in Hom_{\mathcal{A}}^{*}(\mathcal{U})$ and let $\left\{\mathcal{V}_i\right\}_{i \in I}$ represent the $K$-dual $g$-frame sequence for $\left\{\Gamma_i\right\}_{i \in I}$. Suppose $\Xi \in Hom_{\mathcal{A}}^{*}(\mathcal{U})$ is a co-isometry, then the sequence $\left\{\Xi^* \mathcal{V}_i\right\}_{i \in I}$ forms the $K$-dual $g$-frame sequence for $\left\{\Xi^* \Gamma_i\right\}_{i \in I}$.
	\end{theorem}
	
	\begin{proof}
		Given that $\Xi$ is a co-isometry, it implies $\Xi \Xi^* = I$. Thus, for every element $\xi \in \mathcal{U}$, the following equation holds true:
		$$
		\sum_{i \in I}\left(\Xi^* \Gamma_i\right)^*\left(\Xi^* \mathcal{V}_i\right) \xi = \sum_{i \in I} \Gamma_i^* \Xi \Xi^* \mathcal{V}_i \xi = \sum_{i \in I} \Gamma_i^* \mathcal{V}_i \xi = K \xi,
		$$
		which means that the sequence $\left\{\Xi^* \mathcal{V}_i\right\}_{i \in I}$ constitutes the $K$-dual $g$-frame sequence of $\left\{\Xi^* \Gamma_i\right\}_{i \in I}$.
	\end{proof}

	\begin{theorem}
		Assume $K \in Hom_{\mathcal{A}}^{\ast}(\mathcal{U})$ has a closed range, and $\left\{\Gamma_i\right\}_{i \in I}$ forms a $K$-$g$-frame for $\mathcal{U}$  with respect to $\left\{\mathcal{V}_i\right\}_{i \in I}$. In this case, the sequence $\left\{\Gamma_i \pi_{S(Ran(K))}\left(S_{\Gamma}^{-1}\right)^* K\right\}_{i \in I}$ acts as the $K$-dual $g$-frame sequence for $\left\{\Gamma_i \pi_{Ran(K)}\right\}_{i \in I}$, where $S_{\Gamma}$ is the operator defined by
		$$
		S_{\Gamma}: Ran(K) \rightarrow S(Ran(K)).
		$$
	\end{theorem}

	\begin{proof}
		Given that $S_{\Gamma}: Ran(K) \rightarrow S(Ran(K))$ is a bounded operator, it's clear that $\left\{\Gamma_i \pi_{S(Ran(K))}\left(S_{\Gamma}^{-1}\right)^* K\right\}_{i \in I}$ forms a $g$-frame sequence for $\mathcal{U}$. Additionally, considering that $S_{\Gamma}$ possesses the properties of being both self-adjoint and invertible, for every $\xi \in \mathcal{U}$, the following holds true:
		$$
		\begin{aligned}
			K \xi &= \left(S_{\Gamma}^{-1} S_{\Gamma}\right)^* K \xi = S_{\Gamma}^*\left(S_{\Gamma}^{-1}\right)^* K \xi \\
			&= S_{\Gamma} \pi_{S(Ran(K))}\left(S_{\Gamma}^{-1}\right)^* K \xi \\
			&= \sum_{i \in I} \Gamma_i^* \Gamma_i \pi_{S(Ran(K))}\left(S_{\Gamma}^{-1}\right)^* K \xi \\
			&= \sum_{i \in I}\left(\Gamma_i \right)^*\left(\Gamma_i \pi_{S(Ran(K))}\left(S_{\Gamma}^{-1}\right)^* K\right) \xi,
		\end{aligned}
		$$
		which completes the proof, demonstrating the proposed relationship.
	\end{proof}

	\begin{theorem}
		Assume $K \in Hom_{\mathcal{A}}^{\ast}(\mathcal{U})$ and $\left\{\mathcal{V}_i\right\}_{i \in I}$ forms the $K$-dual $g$-frame sequence of $\left\{\Gamma_i\right\}_{i \in I}$. Given $P$ as the $g$-operator for $\left\{\Gamma_i\right\}_{i \in I}$ and $Q$ as the $g$-operator of the $g$-frame sequence $\left\{\Xi_i\right\}_{i \in I}$, the equality $P Q^* = 0$ is both necessary and sufficient for the sequence $\left\{\mathcal{V}_i + \Xi_i\right\}_{i \in I}$ to act as the $K$-dual $g$-frame sequence for $\left\{\Gamma_i\right\}_{i \in I}$.
	\end{theorem}

	\begin{proof}
		If $P Q^* = 0$ holds, and considering $\left\{\mathfrak{E}_i\right\}_{i \in I}$ as the $g$-orthonormal basis for $\mathcal{U}$ with respect to $\left\{\mathcal{V}_i\right\}_{i \in I}$, then for every element $\xi \in \mathcal{U}$, it follows that:
		$$
		\sum_{i \in I} \Gamma_i^* \Xi_i \xi = \sum_{i \in I}\left(\mathfrak{E}_i P^*\right)^*\left(\mathfrak{E}_i Q^*\right) \xi = P \sum_{i \in I} \mathfrak{E}_i^* \mathfrak{E}_i Q^* \xi = P Q^* \xi = 0.
		$$
		Consequently, for all $\xi \in \mathcal{U}$, we have:
		$$
		\sum_{i \in I} \Gamma_i^*\left(\mathcal{V}_i + \Xi_i\right) \xi = \sum_{i \in I} \Gamma_i^* \mathcal{V}_i \xi = K \xi.
		$$
		The reverse implication is derived by following the same logical steps in reverse order, arriving at the initial condition.
	\end{proof}

	\begin{theorem}\label{theo2.18}
		Let $K \in Hom_{\mathcal{A}}^{\ast}(\mathcal{U}),$ and consider two $K$-dual generalized frame sequences $\left\{\Phi_i\right\}_{i \in I}$ and $\left\{\Xi_i\right\}_{i \in I}$ corresponding to the sequence $\left\{\Gamma_i\right\}_{i \in I}$, respectively. Suppose $T_1$ and $T_2$ are linear operators on $\mathcal{U}$ satisfying $T_1+T_2=I$. In this case, the sequence $\left\{\Phi_i T_1+\Xi_i T_2\right\}_{i \in I}$ forms the $K$-dual generalized frame sequence for $\left\{\Gamma_i\right\}_{i \in I}$.
	\end{theorem}

	\begin{proof}
		Given that $T_1 + T_2 = I$, it follows that for every $\xi \in \mathcal{U}$,
		$$
		\begin{aligned}
			\sum_{i \in I} \Gamma_i^*\left(\Phi_i T_1+ \Xi_i T_2\right) \xi &=\sum_{i \in I} \Gamma_i^* \Phi_i T_1 \xi +\sum_{i \in I} \Gamma_i^* \Xi_i T_2 \xi \\
			&=K T_1 \xi +K T_2 \xi =K\left(T_1+T_2\right) \xi =K \xi .
		\end{aligned}
		$$
		This concludes the proof.
	\end{proof}

	\begin{corollary}
		Suppose $K \in Hom_{\mathcal{A}}^{\ast}(\mathcal{U})$ is invertible. Let $\left\{\Phi_i \right\}_{i \in I}$ and $\left\{\Xi_i\right\}_{i \in I}$ be the $K$-dual generalized frame sequences of $\left\{\Gamma_i\right\}_{i \in I}$, respectively. If $T_1$ and $T_2$ are linear operators on $\mathcal{U}$, then the sequence $\left\{\Phi_i T_1+\Xi_i T_2\right\}_{i \in I}$ constitutes the $K$-dual generalized frame sequence of $\left\{\Gamma_i\right\}_{i \in I}$ if and only if $T_1+T_2=I$.
	\end{corollary}
	\begin{proof}
		Assume $T_1+T_2=I$. From Theorem \ref{theo2.18}, $\left\{\Phi_i T_1+\Xi_i T_2\right\}_{i \in I}$ forms the $K$-dual generalized frame. Now, let's establish the sufficiency condition.
		
		If $\left\{\Phi_i T_1+\Xi_i T_2\right\}_{i \in I}$ is the $K$-dual generalized frame sequence for $\left\{\Gamma_i\right\}_{i \in I}$, then for any $\xi \in \mathcal{U}$, we have
		$$
		\begin{aligned}
			K \xi &=\sum_{i \in I} \Gamma_i^*\left(\Phi_i T_1+\Xi_j T_2\right) \xi =\sum_{i \in I}  \Gamma_i^* \Phi_i T_1 \xi+\sum_{i \in I}  \Gamma_i^* \Xi_i T_2 \xi \\
			&=K T_1 \xi+K T_2 \xi=K\left(T_1+T_2\right) \xi
		\end{aligned}
		$$
		Therefore,
		$$K= K(T_1+T_2)$$
		As $K$ is invertible, it follows that $T_1+T_2=I.$
	\end{proof}

	\begin{theorem}
		Let $K \in Hom_{\mathcal{A}}^{\ast}(\mathcal{U})$ such that the range of $K$, $Ran(K)$, is closed. If $\left\{\Gamma_i\right\}_{i \in I}$ forms a $K$-$g$-frame for $\mathcal{U}$ with respect to $\left\{\mathcal{V}_i\right\}_{i \in I}$ with a frame operator $S$, and $C$ and $D$ are its lower and upper bounds, respectively, then
		$$
		C K K^* \leq S \leq D I_{\mathcal{U}}$$ 
	\end{theorem}
	\begin{proof}
		Given that $\left\{\Gamma_i\right\}_{i \in I}$ is a $K$-$g$-frame with the frame operator $S$, it follows that
		\begin{equation}\label{eq2.2}
			\langle S \xi, \xi \rangle=\left\langle\sum_{i \in I} \Gamma_i^* \Gamma_i \xi, \xi \right\rangle=\sum_{i \in I}\left\langle \Gamma_i \xi, \Gamma_i \xi \right\rangle
		\end{equation}
		In line with equation \ref{eq2.2}, we have
		$$
		C \langle K^* \xi, K^* \xi \rangle \leq \langle S \xi, \xi \rangle \leq D \langle \xi, \xi \rangle, \quad \forall \xi \in \mathcal{U}
		$$
		which implies
		$$
		\left\langle C K K^* \xi, \xi \right\rangle \leq \langle S \xi, \xi \rangle \leq \langle B \xi, \xi \rangle.
		$$
	\end{proof}

	\begin{corollary}
		Suppose $K$ is an element of $Hom_{\mathcal{A}}^{\ast}(\mathcal{U})$, and consider $\left\{\Gamma_i\right\}_{i \in I}$ as a g-Bessel sequence within $\mathcal{U}$ with respect to $\left\{\mathcal{V}_i\right\}_{i \in I}$ with the corresponding frame operator denoted as $S$. The sequence $\left\{\Gamma_i\right\}_{i \in I}$ qualifies as a $K$-g-frame for $\mathcal{U}$ if and only if there exists a $K=S^{\frac{1}{2}} Q$ for a certain $Q \in Hom_{\mathcal{A}}^{\ast}(\mathcal{U})$.
	\end{corollary}
	\begin{proof}
		Referencing Lemma \ref{Lemma2.5}, the sequence $\left\{\Gamma_i\right\}_{i \in I}$ can be identified as a $K$-g-frame for $\mathcal{U}$,  with respect to $\left\{\mathcal{V}_i\right\}_{i \in I}$, if and only if there exists a positive constant $C$ that ensures
		$$
		C K K^* \leq S = S^{\frac{1}{2}}\left(S^{\frac{1}{2}}\right)^*.
		$$
		Following this, as per Lemma \ref{Dougl1}, it is concluded that a bounded operator $U$ can be found such that
		$$
		K = S^{\frac{1}{2}} U.
		$$
	\end{proof}

	\begin{theorem}\label{theo2.22}
		Consider a mapping $K \in Hom_{\mathcal{A}}^{\ast}(\mathcal{U})$, and  $\left\{\Gamma_i\right\}_{i \in I}$  a $K$-g-frame for $\mathcal{U}$, with respect to $\left\{\mathcal{V}_i\right\}_{i \in I}$, having bounds $C$ and $D$ as its lower and upper limits, respectively. Assume $Q$ is an element of $Hom_{\mathcal{A}}^{\ast}(\mathcal{U})$ with the property of a closed range and the condition $Q K = K Q$ being met. Under these circumstances, the following are observed:
		\begin{enumerate}
			\item The sequence $\left\{\Gamma_i Q^*\right\}_{i \in I}$ establishes itself as a $K$-g-frame for the range of $Q$ (denoted as $Ran(Q)$) with respect to $\left\{\mathcal{V}_i\right\}_{i \in I}$, bounded below and above by $C \left\|Q^{\dagger}\right\|_{\infty}^{-2}$ and $D \|Q\|_{\infty}^2$, respectively.
			\item The frame operator $S$ corresponding to $\left\{\Gamma_i Q^*\right\}_{i \in I}$ fulfills the relation $S = Q S_{\Gamma} Q^*$, with $S_{\Gamma}$ being the frame operator for $\left\{\Gamma_i\right\}_{i \in I}$.
		\end{enumerate}
	\end{theorem}

	\begin{proof}
		Considering that $Q$ has a closed range, it consequently possesses a pseudo-inverse $Q^{\dagger}$, fulfilling $Q Q^{\dagger} = I_{Ran(Q)}$, where the subscript $Ran(Q)$ of $Q^{\dagger}$ is omitted. Now, $I_{Ran(Q)} = I_{Ran(Q)}^* = \left(Q^{\dagger}\right)^*\left(Q^*\right)$. Thus, for every $\xi \in Ran(Q)$,
		$$
		K^* \xi = \left(Q^{\dagger}\right)^* Q^* K^* \xi
		$$
		leads to,
		$$
		\begin{aligned}
			\langle K^* \xi, K^* \xi\rangle &= \langle \left(Q^{\dagger}\right)^* Q^* K^* \xi , \left(Q^{\dagger}\right)^* Q^* K^* \xi \rangle \\ 
			&\leq \left\|\left(Q^{\dagger}\right)^*\right\|_{\infty}^{2} \langle Q^* K^* \xi , Q^* K^* \xi \rangle \\
			&= \left\|\left(Q^{\dagger}\right)^*\right\|_{\infty}^{2} \langle K^* Q^* \xi , K^* Q^* \xi \rangle.
		\end{aligned}
		$$
		Consequently,
		$$
		\left\|\left(Q^{\dagger}\right)^*\right\|_{\infty}^{-2} \langle K^* \xi, K^* \xi\rangle \leq \langle K^* Q^* \xi , K^* Q^* \xi \rangle.
		$$
		For each $\xi \in Ran(Q)$, we have
		\begin{equation}\label{eq1}
			\sum_{i \in I} \langle \Gamma_i Q^* \xi , \Gamma_i Q^* \xi \rangle \geq C \langle K^* Q^* \xi, K^* Q^* \xi \rangle \geq C \left\|\left(Q^{\dagger}\right)^*\right\|_{\infty}^{-2} \langle K^* \xi, K^*. \xi\rangle
		\end{equation}
		Moreover, $\left\{\Gamma_i Q^*\right\}_{i \in I}$ is a g-Bessel sequence, as evidenced by:
		\begin{equation}\label{eq2}
			\sum_{i \in I} \langle \Gamma_i Q^* \xi, \Gamma_i Q^* \xi\rangle \leq D \langle Q^* \xi, Q^* \xi \rangle \leq B \|Q\|_{\infty}^2 \langle \xi, \xi \rangle,
		\end{equation}
		From \ref{eq1} and \ref{eq2}, point 1 is established.
		For point 2, it is clear that:
		$$
		\begin{aligned}
			S \xi &= \sum_{i \in I} \left(\Gamma_i Q^*\right)^* \Gamma_j Q^* \xi \\
			&= Q \sum_{i \in I} \Gamma_i^* \Gamma_i Q^* \xi \\
			&= Q S_{\Lambda} Q^* \xi.
		\end{aligned}
		$$
		This concludes the proof of Theorem \ref{theo2.22}.
	\end{proof}

	\begin{theorem}
		Assume we have a mapping $K \in Hom_{\mathcal{A}}^{\ast}(\mathcal{U})$ and  $\left\{\Gamma_i\right\}_{i \in I}$ as a $K$-$g$-frame for $\mathcal{U}$,  with respect to $\left\{\mathcal{V}_i\right\}_{i \in I}$. Provided that $Q$, also a member of $Hom_{\mathcal{A}}^{\ast}(\mathcal{U})$, functions as a co-isometry and satisfies the commutative relation $Q K = K Q$, it follows that the set $\left\{\Gamma_i Q^*\right\}_{i \in I}$ forms a $K$-g-frame for $\mathcal{U}$  with respect to  $\left\{\mathcal{V}_i\right\}_{i \in I}$.
	\end{theorem}

	\begin{proof}
		Given that $Q$ is a co-isometry and satisfies $Q K = K Q$, for any $\xi \in \mathcal{U}$, it follows that
		$$
		\langle K^* Q^* \xi, K^* Q^* \xi\rangle = \langle Q^* K^* \xi, Q^* K^* \xi\rangle = \langle K^* \xi, K^* \xi\rangle .
		$$
		Assuming $C$ and $D$ are the respective lower and upper bounds of $\left\{\Gamma_i Q^*\right\}_{i \in I}$, then for every $\xi \in \mathcal{U}$, we have
		$$
		\sum_{i \in I} \langle \Gamma_i Q^* \xi , \Gamma_i Q^* \xi\rangle \geq C \langle K^* Q^* \xi, K^* Q^* \xi\rangle = C \langle K^* \xi, K^* \xi\rangle .
		$$
		Conversely,
		$$
		\sum_{i \in I} \langle \Gamma_i Q^* \xi , \Gamma_i Q^* \xi\rangle \leq D \langle Q^* \xi , Q^* \xi\rangle \leq B\left\|Q\right\|_{\infty}^2 \langle \xi , \xi\rangle.
		$$
		Therefore, $\left\{\Gamma_i Q^*\right\}_{i \in I}$ constitutes a $K$-g-frame for $\mathcal{U}$ with respect to $\left\{\mathcal{V}_i\right\}_{i \in I}$.
	\end{proof}

	\begin{corollary}
		Assume $K$ is a part of $Hom_{\mathcal{A}}^{*}(\mathcal{U})$ and $\left\{\Gamma_i\right\}_{i \in I}$ constitutes a $K$-g-frame for $\mathcal{U}$, with respect to $\left\{\mathcal{V}_i\right\}_{i \in I}$. If $Q$ operates as an isometry, then the set $\left\{Q \Gamma_i\right\}_{i \in I}$ also forms a $K$-g-frame for $\mathcal{U}$,  with respect to $\left\{\mathfrak{E} \mathcal{V}_i\right\}_{i \in I}$, maintaining the same frame bounds as $\left\{\Gamma_i\right\}_{i \in I}$.
	\end{corollary}

	\begin{proof}
		Let's assume $C$ and $D$ are, respectively, the lower and upper bounds for the set $\left\{\Gamma_i\right\}_{i \in I}$. Consequently, it follows that
		$$
		C \langle K^* \xi, K^* \xi \rangle \leq \sum_{i \in I} \langle Q \Gamma_i \xi, Q \Gamma_i \xi\rangle \leq D \langle \xi, \xi \rangle.
		$$
		It is important to note here that $Q$ functions as an isometry.
	\end{proof}
	
	\begin{theorem}
		Consider $K$ as an element of $Hom_{\mathcal{A}}^{*}(\mathcal{U})$ and $\left\{\Gamma_i\right\}_{i \in I}$ as a g-Bessel sequence for $\mathcal{U}$,  with respect to $\left\{\mathcal{V}_i\right\}_{i \in I}$. Let's assume $T$ represents the synthesis operator for $\left\{\Gamma_i\right\}_{i \in I}$. The condition $Ran(K) = Ran(T)$ holds true if and only if $\left\{\Gamma_i\right\}_{i \in I}$ forms a tight $K$-g-frame for $\mathcal{U}$  with respect to $\left\{\mathcal{V}_i\right\}_{i \in I}$.
	\end{theorem}

	\begin{proof}
		Let us start by assuming $Ran(K) = Ran(T)$. Based on Lemma \ref{Dougl1}, there exists a constant $A > 0$ such that
		$$
		A K K^* = T T^*,
		$$
		leading to 
		$$
		A \langle K^* \xi, K^* \xi \rangle = \langle T^* \xi, T^* \xi \rangle = \sum_{i \in I} \langle \Gamma_i \xi, \Gamma_i \xi \rangle.
		$$
		
		Conversely, if $\left\{\Gamma_i\right\}_{i \in I}$ is recognized as a tight $K$-g-frame with a frame bound of $A$, then it follows that
		$$
		A \langle K^* \xi, K^* \xi \rangle = \sum_{i \in I} \langle \Gamma_i \xi, \Gamma_i \xi \rangle = \langle T^* \xi, T^* \xi \rangle,
		$$
		which simplifies to
		$$
		A K K^* = T T^*.
		$$
		Therefore, invoking Lemma \ref{Dougl1} again, it can be concluded that $Ran(K) = Ran(T)$.
	\end{proof}
	
	\begin{theorem}
		Let $\left\{\Gamma_i\right\}_{i \in I}$ and $\left\{\Xi_i\right\}_{i \in I}$ are two g-Bessel sequences for $\mathcal{X}$ with respect to $\left\{\mathcal{Y}_i\right\}_{i \in I}$ with bounds $A$ and $B$, respectively. $T_1$ and $T_2$ are synthesis operators of $\left\{\Gamma_i\right\}_{i \in I}$ and $\left\{\Xi_i\right\}_{i \in I}$, respectively. If $T_1 T_2^*=I_{\mathcal{X}}$, then $\left\{\Gamma_i\right\}_{i \in I}$ and $\left\{\Xi_i\right\}_{i \in I}$ are two $K$-g-frames for $\mathcal{X}$ with respect to $\left\{\mathcal{Y}_i\right\}_{i \in I}$.
		
	\end{theorem}

	\begin{theorem}
		Suppose $\left\{\Gamma_i\right\}_{i \in I}$ and $\left\{\Xi_i\right\}_{i \in I}$ represent two g-Bessel sequences for $\mathcal{U}$, each associated with $\left\{\mathcal{V}_i\right\}_{i \in I}$ and having bounds $A$ and $B$, respectively. Let $T_1$ and $T_2$ be the synthesis operators for $\left\{\Gamma_i\right\}_{i in I}$ and $\left\{\Xi_i\right\}_{i in I}$, respectively. If it holds that $T_1 T_2^* = I_{\mathcal{U}}$, then both $\left\{\Gamma_i\right\}_{i \in I}$ and $\left\{\Xi_i\right\}_{i \in I}$ qualify as $K$-g-frames for $\mathcal{U}$ with respect to $\left\{\mathcal{V}_i\right\}_{i \in I}$.
	\end{theorem}

	\section{ Characterizing the $K$-$g$ Frames by Other Concepts}
	In this section, we discuss the concept of identity resolution and explore the idea of quotients derived from bounded operators. These concepts will subsequently be utilized in the construction of $K$-g-frames.

	\begin{definition}
		Let $I$ be an indexing set. A collection of bounded operators $\left\{\Psi_i\right\}_{i \in I}$ on $\mathcal{U}$ is defined as an (unconditional) resolution of the identity on $\mathcal{U}$ if, for every element $\xi \in \mathcal{U}$, it satisfies:
		$$
		\xi = \sum_{i \in I} \Psi_i(\xi)
		$$
		(where the series converges unconditionally for each $\xi \ in \mathcal{U}$). It can be readily shown that if $\left\{\Theta_{i}\right\}_{i \in I}$ and $\left\{\Psi_{i}\right\}_{i \in I}$ are both resolutions of the identity on $\mathcal{U}$, then the combined set $\left\{\Theta_{i} \Psi_{i}\right\}_{i, i \in I}$ also constitutes a resolution of the identity on $\mathcal{U}$.
	\end{definition}

	\begin{theorem}
		Consider an element $K$ in $Hom_{\mathcal{A}}^*(\mathcal{U})$, which is invertible and where both $K$ and its inverse $K^{-1}$ are uniformly bounded. Assume that the collection $\left\{\Psi_i\right\}_{i \in I}$ forms the identity resolution of $\mathcal{U}$ and constitutes a $g$-Bessel sequence with a bound of $D$. Under these conditions, $\left\{\Psi_i\right\}_{i \in I}$ is identified as a $K$-$g$-frame for $\mathcal{U}$ in relation to $\left\{\mathcal{V}_i\right\}_{i \in I}$.
	\end{theorem}

	\begin{proof}
		Given that $\left\{\Psi_i\right\}_{i \in I}$ constitutes an identity resolution for the space $\mathcal{U}$, we observe that for every element $\xi$ within $\mathcal{U}$, the expression $\xi=\sum_{i \in I} \Psi_i(\xi)$ is valid. Consequently, the following relationship is established:
		$$
		K^* \xi = \sum_{i \in I} \Psi_i K^* \xi,
		$$
		leading to the inference that
		$$
		\langle K^* \xi , K^* \xi \rangle = \langle \sum_{i \in I} \Psi_i K^* \xi, \sum_{i \in I} \Psi_i K^* \xi \rangle \leq \sum_{i \in I} \langle \Psi_i K^* \xi,  \Psi_i K^* \xi \rangle \leq D \langle  K^* \xi,   K^* \xi \rangle \leq D\|K\|_{\infty}^2\langle \xi, \xi \rangle.
		$$
	\end{proof}

	Let $\mathcal{F}, \mathcal{T} \in Hom_{\mathcal{A}}^*(\mathcal{U})$. The quotient $[\mathcal{F}/ \mathcal{T}]$ is a map from $Ran(\mathcal{F})$ to $Ran(\mathcal{T})$ defined by $\mathcal{T} x \mapsto \mathcal{F} x$. Similar to the case in Hilbert spaces, $\mathcal{Q} = [\mathcal{F} / \mathcal{T}]$ is a linear operator on $\mathcal{U}$ if and only if $Ker(\mathcal{T}) \subset Ker(\mathcal{F})$.

	\begin{theorem}
		Suppose $K \in Hom_{\mathcal{A}}^*(\mathcal{U})$, and consider  $\left\{\Psi_i\right\}_{i \in I}$, which constitutes a g-Bessel sequence in the context of $\mathcal{U}$ with respect to $\left\{\mathcal{V}_i\right\}_{i \in I}$. Given that $S$ represents the frame operator for $\left\{\Psi_i\right\}_{i \in I}$, it follows that $\left\{\Psi_i\right\}_{i \in I}$ forms a $K$-g-frame with respect to $\left\{\mathcal{V}_i\right\}_{i \in I}$ if and solely if the operator defined by $\left[K^* / S^{\frac{1}{2}}\right]$ demonstrates boundedness.
	\end{theorem}

	\begin{proof}
		Given that $\left\{\Psi_i\right\}_{i \in I}$ constitutes a $K$-g-frame, we can ascertain the existence of a constant $C>0$. This ensures that for any $\xi \in \mathcal{U}$, the following inequality holds:
		
		$$
		C\langle K^* \xi , K^* \xi \rangle \leq \sum_{i \in I}\langle \Psi_i \xi, \Psi_i \xi  \rangle=\langle Se \xi, \xi\rangle=\langle Se^{\frac{1}{2}} \xi , Se^{\frac{1}{2}} \xi \rangle
		$$
		holds. This implies that
		$$
		C\langle K^* \xi ,  K^* \xi\rangle \leq\langle Se^{\frac{1}{2}} \xi , Se^{\frac{1}{2}} \xi \rangle.
		$$
		Define the operator $\Xi: Ran\left(Se^{\frac{1}{2}}\right) \rightarrow Ran\left(K^*\right)$ by $\Xi\left(Se^{\frac{1}{2}} \xi\right)=K^* \xi$ for all $\xi \in \mathcal{U}$. The operator $\Xi$ is well-defined since $Ker\left(Se^{\frac{1}{2}}\right) \subseteq Ker\left(K^*\right)$. Consequently, we obtain
		$$
		\langle \Xi Se^{\frac{1}{2}} \xi, \Xi Se^{\frac{1}{2}} \xi \rangle=\langle K^* \xi,K^* \xi\rangle \leq \frac{1}{\sqrt{C}} \langle Se^{\frac{1}{2}} \xi , Se^{\frac{1}{2}} \xi \rangle.
		$$
		Thus, $\Xi$ is bounded. By the concept of quotient of bounded operators, $\Xi$ is expressed as $\left[K^* / Se^{\frac{1}{2}}\right]$.
		
		On the other hand, if $\left[K^* / Se^{\frac{1}{2}}\right]$ is bounded, there exists a constant $D>0$ such that for all $\xi \in \mathcal{U}$,
		$$
		\langle K^* \xi , K^* \xi \rangle \leq D\langle Se^{\frac{1}{2}} \xi ,Se^{\frac{1}{2}} \xi\rangle,
		$$
		implying
		$$
		\frac{1}{D}\langle K^* \xi, K^* \xi\rangle ^2 \leq\langle Se^{\frac{1}{2}} \xi , Se^{\frac{1}{2}} \xi\rangle =\sum_{i \in I}\langle \Psi_i \xi, \Psi_i \xi\rangle.
		$$
		Therefore, $\left\{\Psi_i\right\}_{i \in I}$ is a $K$-g-frame for $\mathcal{U}$ with respect to $\left\{\Theta_i\right\}_{i in I}$.
	\end{proof}
	
	\begin{theorem}
		Let $K \in Hom_{\mathcal{A}}^{*}(\mathcal{X})$ and $\left\{\Psi_i\right\}_{i \in I}$ be a $g$-Bessel sequence for $\mathcal{X}$ with respect to $\left\{\mathcal{Y}_i\right\}_{i \in I}$ with frame operator $S$. If $Q \in Hom_{\mathcal{A}}^{*}(\mathcal{X})$, then the following statements are equivalent:
		\begin{enumerate}
			\item $\left[K^* / S^{\frac{1}{2}} Q\right]$ is bounded.
			\item $\left\{\Psi_i Q \right\}_{i \in I}$ is a $K$-g-frame for $\mathcal{X}$ with respect to $\left\{\mathcal{Y}_i \right\}_{i \in I}$;
			
		\end{enumerate}
	\end{theorem}

	\begin{theorem}
		Consider $K \in Hom_{\mathcal{A}}^{*}(\mathcal{X})$ and let $\left\{\Psi_i\right\}_{i \in I}$ constitute a $g$-Bessel sequence in the context of $\mathcal{U}$, associated with $\left\{\mathcal{V}_i\right\}_{i in I}$, and possessing the frame operator $S$. Given another element $Q$ within $Hom_{\mathcal{A}}^{*}(\mathcal{V})$, the following conditions are equivalent:
		\begin{enumerate}
			\item The operator $\left[K^* / S^{\frac{1}{2}} Q\right]$ exhibits boundedness.
			\item  $\left\{\Psi_i Q \right\}_{i \in I}$ forms a $K$-g-frame for $\mathcal{U}$ with respect to $\left\{\mathcal{V}_i \right\}_{i \in I}$.
		\end{enumerate}
	\end{theorem}

	\begin{proof}
		$(2) \Rightarrow (1)$ Assume $C$ is the lower bound constant for the $K$-$g$-frame $\left\{\Psi_i Q\right\}_{i \in I}$ for $\mathcal{U}$ with respect to $\left\{\Theta_i\right\}_{i \in I}$. Then,
		$$
		C\langle K^* \xi, K^* \xi \rangle \leq \sum_{i \in I}\langle \Psi_i Q \xi, \Psi_i Q \xi \rangle .
		$$
		Furthermore,
		$$
		\sum_{i \in I}\langle \Psi_i Q \xi,  \Psi_i Q \xi \rangle=\langle S Q \xi, Q \xi \rangle=\langle S^{\frac{1}{2}} Q \xi , S^{\frac{1}{2}} Q \xi\rangle,
		$$
		since $\left\{\Psi_i\right\}_{i \in I}$ is a g-Bessel sequence for $\mathcal{U}$ and $S$ its frame operator. Therefore,
		$$
		C\langle K^* \xi, K^* \xi \rangle \leq \langle S^{\frac{1}{2}} Q \xi , S^{\frac{1}{2}} Q \xi \rangle,
		$$
		indicating that $\left[K^* / S^{\frac{1}{2}} Q\right]$ is bounded. \\
		$(1) \Rightarrow (2)$ If $\left[K^* / S^{\frac{1}{2}} Q\right]$ is bounded, there exists a constant $\beta>0$ such that
		$$
		\langle K^* \xi, K^* \xi \rangle \leq \beta \langle S^{\frac{1}{2}} Q \xi, S^{\frac{1}{2}} Q \xi \rangle.
		$$
		Assuming that the g-Bessel sequence $\left\{\Psi_i Q\right\}_{i \in I}$ for $\mathcal{U}$ with respect to $\left\{\Theta_i\right\}_{i \in I}$ has an upper bound constant $D$, it follows
		$$
		\frac{1}{\beta}\langle K^* \xi ,K^* \xi \rangle \leq\langle S^{\frac{1}{2}} Q \xi, S^{\frac{1}{2}} Q \xi \rangle =\sum_{i \in I}\langle \Psi_i Q \xi, \Psi_i Q \xi \rangle  \leq D\|\mathfrak{E}\|_{\infty}^2\langle \xi, \xi \rangle,
		$$
		leading to the conclusion.
	\end{proof}

	\section*{Availablity of data and materials}
	Not applicable.

	\section*{Conflict of interest}
	The authors declare that they have no competing interests.

	\section*{Fundings}
	The authors declare that there is no funding available for this paper.


\begin{thebibliography}{99}
		\bibitem{Alizadeh} L. Alizadeh, M. Hassani, On frames for countably generated Hilbert modules over locally $C^{\ast}$-algebras, Commun. Korean Math. Soc. 33 (2018) 527–533. https://doi.org/10.4134/CKMS.c170194.
		\bibitem{Assila} N. Assila, H. Labrigui, A. Touri, M. Rossafi, Integral operator frames on Hilbert $C^{\ast}$-modules, Ann. Dell’Universita Di Ferrara (2024). https://doi.org/10.1007/s11565-024-00501-z.
		\bibitem{Azhini} M. Azhini, N. Haddadzadeh, Fusion frames in Hilbert modules over proC*-algebras, Int. J. Ind. Math. 5 (2013) 109–118.
		
		\bibitem{Rouma1} R. El Jazzar, R. Mohamed, On Frames in Hilbert Modules Over Locally $C^{\ast}$-Algebras, Int. J. Anal. Appl. 21 (2023) 130. https://doi.org/10.28924/2291-8639-21-2023-130.
		\bibitem{Frago} M. Fragoulopoulou, Tensor products of enveloping locally $C^{\ast}$-algebras, Schriftenreihe, Univ. M\"{u}nster, (1997) 1-81.
		
		\bibitem{Frag} M. Fragoulopoulou, An introduction of the representation theory of topological $\ast$-algebras, Schriftenreihe, Univ. M\"{u}nster, 48 (1988) 1–81. 
		\bibitem{Ghiati}  M. Ghiati, M. Rossafi, M. Mouniane, H. Labrigui, A. Touri, Controlled continuous $\ast$-g-frames in Hilbert $C^{\ast}$-modules, J. Pseudo-Differential Oper. Appl. 15 (2024). 
		\bibitem{Haddad}
		N. Haddadzadeh, G-FRAMES IN HILBERT PRO-C*-MODULES, Int. J. Pure Apllied Math. 105 (2015) 727–743. https://doi.org/10.12732/ijpam.v105i4.13.
		\bibitem{Inoue} A. Inoue, Locally $C^{\ast}$-algebra, Mem. Fac. Sci. Kyushu Univ. Ser. A, Math. 25 (1972) 197–235. https://doi.org/10.2206/kyushumfs.25.197.
		\bibitem{Mallios} A. Mallios, Topological Algebras: Selected Topics, North Holland, Elsevier, 2011.
		\bibitem{Massit} H. Massit, M. Rossafi, C. Park, Some relations between continuous generalized frames, Afrika Mat. 35 (2024) 12. https://doi.org/10.1007/s13370-023-01157-2.
		\bibitem{Philip} N. C. Phillips, Inverse limits of $C^{\ast}$-algebras, J. Oper. Theory 19 (1988) 159–195.
		\bibitem{Philips} N. C. Phillips, Representable $K$-Theory for $\sigma$-$C^{\ast}$-algebras, $K$-Theory 3 (1989) 441–478. 
		\bibitem{Rouma} M. Rossafi, R. El Jazzar, R. Mohapatra, Douglas'factorization theorem and atomic system in Hilbert pro$-C^{\ast}$-module, Sahand Commun. Math. Anal. 21 (2024) 25–49.
		\bibitem{Rossafi} M. Rossafi, F.-D. Nhari, C. Park, S. Kabbaj, Continuous g-Frames with $C^{\ast}$-Valued Bounds and Their Properties, Complex Anal. Oper. Theory 16 (2022) 44. 
		
		
	\end{thebibliography}
	\end{document}